\documentclass[12pt,verbatim]{amsart}
\usepackage{amsfonts}
\usepackage{amsthm}
\usepackage{amsmath}
\usepackage{amssymb}
\usepackage{amscd}
\usepackage{subfigure}

\usepackage{verbatim}
\usepackage{ifthen}
\usepackage{graphicx}
\usepackage{caption}

\newcounter{minutes}
\divide\time by 60
\newcounter{hours}
\multiply\time by 60 \addtocounter{minutes}{-\time}

\theoremstyle{plain}
\newtheorem{thm}[equation]{Theorem}
\newtheorem{cor}[equation]{Corollary}
\newtheorem{lem}[equation]{Lemma}

\theoremstyle{definition}
\newtheorem{defn}[equation]{Definition}

\newtheorem{rem}[equation]{Remark}

\newtheorem{nonsec}[equation]{}

\newcommand{\R}{{\mathbb R}}
\newcommand{\Rn}{{\R}^n}

\newcommand{\B}{{{\mathbb B}^2}}
\newcommand{\Bn}{{{\mathbb B}^n}}

\newcommand{\Hn}{{\mathbb H}^n}

\newcommand{\bdr}{\partial}
\newcommand{\ang}{\measuredangle}
\newcommand{\K}{{\mathcal K}}
\newcommand{\E}{{\mathcal E}}
\def\be{\begin{equation}}
\def\ee{\end{equation}}

\def\tang{\mathrel{\hbox{\rlap{%
\kern 0.2ex\hbox to 1.8ex{\hrulefill}}\raise2.6pt\hbox{$\bigcirc$}}}}

\DeclareMathOperator*{\arth}{arth}     
\DeclareMathOperator{\ray}{ray}

\DeclareMathOperator{\comp}{comp}

\numberwithin{equation}{section}

\begin{document}

\def\thefootnote{}
\footnotetext{ \texttt{\tiny File:~\jobname .tex,
          printed: \number\year-\number\month-\number\day,
          \thehours.\ifnum\theminutes<10{0}\fi\theminutes}
} \makeatletter\def\thefootnote{\@arabic\c@footnote}\makeatother

\title{The visual angle metric and quasiregular maps}

\author{Gendi Wang}
\address{School of Sciences, Zhejiang Sci-Tech University, Hangzhou 310018, China}
\email{gendi.wang@zstu.edu.cn}
\author{Matti Vuorinen}
\address{Department of Mathematics and Statistics, University of Turku, Turku 20014, Finland}
\email{vuorinen@utu.fi}

\maketitle

\begin{abstract}
The distortion of distances between points under maps is studied. We first prove a Schwarz-type lemma for quasiregular maps of the unit disk involving the visual angle metric. Then we investigate conversely the quasiconformality of a bilipschitz map with respect to the visual angle metric on convex domains. For the unit ball or half space, we prove that a  bilipschitz map with respect to the visual angle metric is also bilipschitz with respect to the hyperbolic metric.
We also obtain various inequalities relating the visual angle metric to other metrics such as the distance ratio metric and the quasihyperbolic metric.

\end{abstract}



\section{Introduction}

One of the main problems of geometric function theory is to study the way in which maps distort distances between points. The standard method to discuss this problem is to study uniform continuity of maps between suitable metric spaces. For example, one of the cornerstones of geometric function theory, the Schwarz Lemma originally formulated for bounded analytic functions of the unit disk, has been extended to several other classes of functions and to several metrics other than the Euclidean metric.

Let $G \subsetneq \Rn$ be a domain and $x,y \in G$. The {\it visual angle metric} is defined as
$$v_G(x,y)=\sup_{z\in\partial G}\ang(x,z,y)\,,$$
where $\partial G$ is not a proper subset of a line.
This paper is based on our earlier paper  \cite{klvw}, where we introduced this metric and studied
some of the basic properties of the visual angle metric.
In particular, we gave some estimates for the visual angle metric
in terms of the hyperbolic metric in the case when the domain is either the unit ball or the upper half space.
Our goal here is to study how the visual angle metric behaves under quasiconformal maps. The main result is the following theorem.
\begin{thm}\label{vs}
If $f:\B\rightarrow \R^2$ is a non-constant $K$-quasiregular map with $f\B\subset\B$, then
$$v_{\B}(f(x),f(y))\le C(K) \max\{v_{\B}(x,y),\,v_{\B}(x,y)^{1/K}\}$$
for all $x,\,y\in\B$, where 
$C(K)=2\cdot 4^{1-1/K}$ and $C(1)=2$.
\end{thm}

\begin{rem}
It is clear that the visual angle metric is similarity invariant. It is not difficult to show that it fails to be  M\"obius invariant. However, by Lemma \ref{mthm1} the visual angle metric is not changed by more than a factor $2$ under the M\"obius transformations of $G$ onto $G'$ for $G\,,G'\in\{\Bn,\Hn\}$.
For M\"obius transformations of the unit disk onto itself, we know
that the best constant in place of $C(K)$ is 2 by \cite[Theorem 1.2]{klvw}. Therefore, we see that
the constant $C(K)$ is asymptotically sharp when $K \to 1\,.$

\end{rem}

Moreover, we prove


\begin{thm}\label{mthf}
Let $G_1\,, G_2$ be proper convex subdomains of $\Rn$. Let $f: G_1\rightarrow G_2=f(G_1)$ be an $L$-bilipschitz map with respect to the visual angle metric. Then $f$ is quasiconformal and with linear dilatation at most $4L^2$.
\end{thm}

\begin{thm}\label{mths}
 For $G_1, G_2\in \{\Bn, \Hn\}$, let $f: G_1\rightarrow G_2$ be a bilipschitz map with respect to the visual angle metric. Then $f$ is also a bilipschitz map with respect to the hyperbolic metric.
\end{thm}
We also prove various inequalities relating the visual angle metric to other metrics such as the distance ratio metric and the quasihyperbolic metric.
Agard and Gehring \cite[Theorems 2 and 3 ]{ag} have studied the transformation of angles under quasiconformal maps. Our results
differ from their work, because in our case the vertex of the angle is on the boundary of the domain of definition of the mapping.

\bigskip

\section{Preliminaries}
\begin{nonsec}{\bf Notation.}
Throughout this paper we will discuss domains $G \subset \mathbb{R}^n$, i.e.,
open and connected subsets of $\mathbb{R}^n$. For $x,y \in G$, the Euclidean distance between
$x$ and $y$ is denoted by $|x-y|$ or $d(x,y)$, as usual. The notation
$d(x,\bdr G)$ or $d(x)$ for abbreviation, stands for the distance from the point $x$ to the boundary $\bdr G$ of the domain $G$.

The Euclidean $n$-dimensional
ball with center $z$ and radius $r$ is denoted by $\Bn(z,r)$, and its boundary
sphere by $S^{n-1}(z,r)$. In particular, $\Bn(r)=\Bn(0,r),\;
S^{n-1}(r)=S^{n-1}(0,r)$, and $\Bn = \Bn(0,1),\; S^{n-1}=S^{n-1}(0,1)$.

The upper Lobachevsky-Poincar\'e $n$-dimensional half space (as a set) is denoted
by $\Hn = \{(z_1,z_2,\cdots, z_n) \in \Rn \,: \,  z_n > 0\}$. For $t\in \R$ and $a\in\Rn\setminus\{0\}$, we denote a hyperplane in $\overline{\Rn}=\Rn\cup\{\infty\}$ by $P(a,t)=\{x\in\Rn: x\cdot a=t\}\cup\{\infty\}$.

Given a vector $u\in\Rn\setminus\{0\}$ and a point $x \in \Rn$, the line passing through $x$ with direction vector $u$ is denoted by $L(x,u)=\{x+tu\ :\ t \in \R\}$ and the ray starting at $x$
with direction $u$ is $\ray(x,u)=\{x+tu\ :\ t>0\}\,.$
Given two points $x$ and $y$, the segment between them is denoted by
$[x,y]=\{(1-t)x+ty\;:\; 0\le t\le1\}$.

Given three distinct points $x,y$ and $z \in \Rn$, the notation $\ang(x,z,y)$
means the angle in the range $[0,\pi]$ between the
segments $[x,z]$ and $[y,z]$.

Let $G$ be a set for which a metric $m_G$ is defined. We define the open $m$-ball $B_m(x,t)$ with center $x$ and radius $t$, and the corresponding boundary $m$-sphere $S_m(x,t)$, in $m$-metric, as the set
$$B_m(x,t)=\{z\in G: m_G(x,z)<t\}$$
and
$$S_m(x,t)=\{z\in G: m_G(x,z)=t\},$$
respectively.
\end{nonsec}

\medskip

\begin{nonsec}{\bf Visual angle metric.}
The visual angle metric has the following monotonicity property:
if $G_1\,,G_2$ are proper subdomains of $\Rn$ and $x\,,y \in G_1 \subset G_2$, then $v_{G_1}(x,y) \ge v_{G_2}(x,y)$.


In the unit ball, for the special case $x\neq0\,,y=0$, we have the following useful formula
\begin{eqnarray}\label{omega0x}
v_{\Bn}(0,x) = \arcsin |x|\in (0,\pi/2),
\end{eqnarray}
and for $|x|=|y|\neq 0$, $\theta=\frac 12 \ang(x,0,y)\in (0, \pi/2]$, we have
\begin{eqnarray}\label{omega1x}
v_{\Bn}(x,y) = 2\arctan\frac{|x|\sin\theta}{1-|x|\cos\theta}.
\end{eqnarray}

\begin{figure}[h]
\subfigure[]{\includegraphics[width=.28 \textwidth]{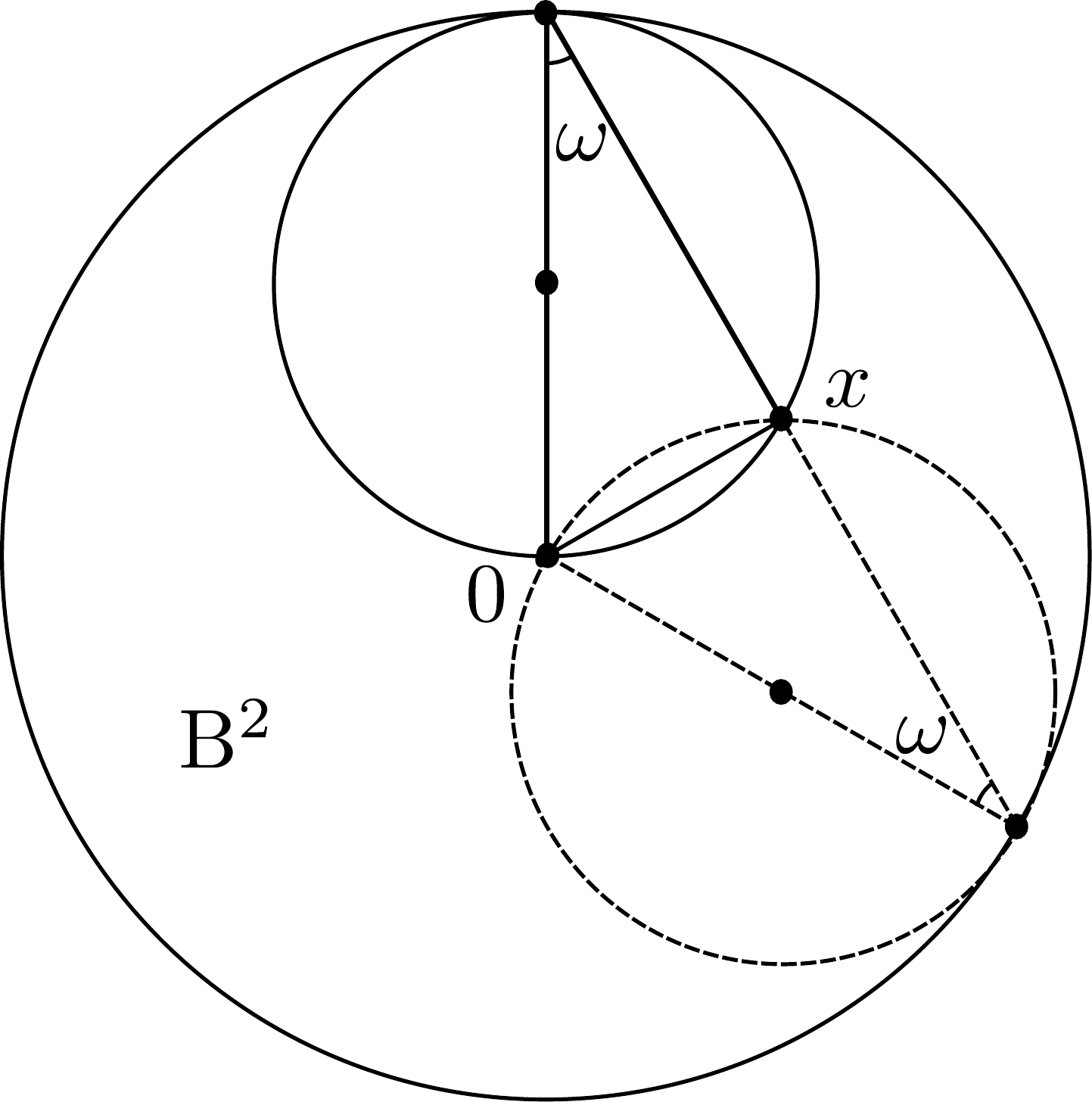}}
\hspace{.1 \textwidth}
\subfigure[]{\includegraphics[width=.28 \textwidth]{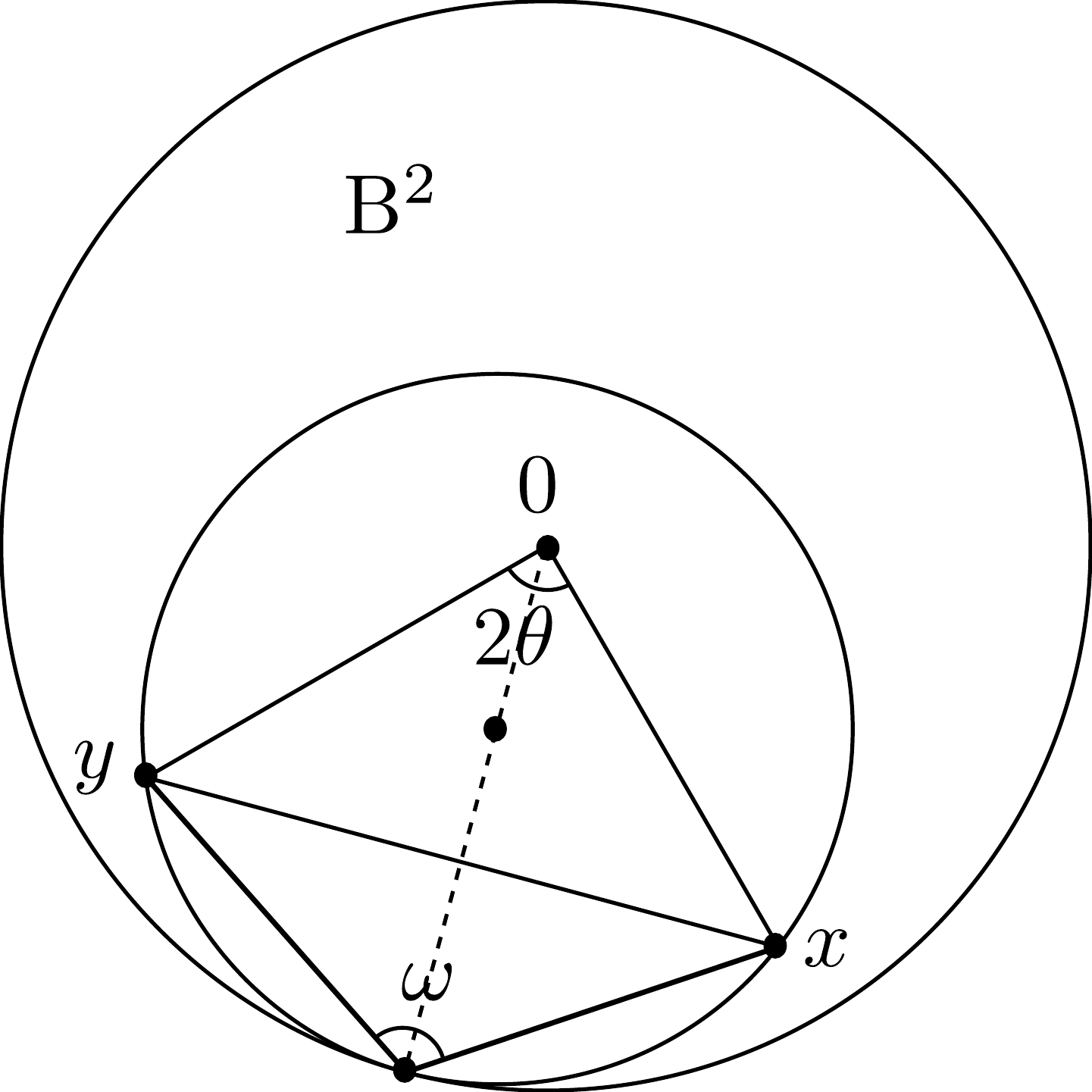}}
\caption{\label{angmetfig34}  Figure in $\B$: (a) Special case \eqref{omega0x}, where $y=0$. (b) Special case \eqref{omega1x}, where $|x|=|y|$ and $\ang(x,0,y)=2\theta$.}
\end{figure}
\end{nonsec}

It is well-known that the diameter of a Euclidean sphere is twice of its radius, but this is not the case for the visual angle metric in $\Bn$.

\begin{rem}\label{red}
Choose $x\in(0, e_1)$, $M=v_{\Bn}(0,x)\in (0,\pi/2)$. Then by (\ref{omega0x}), $B_v(0,M)$ is a Euclidean ball with radius $|x|=\sin M$.
By (\ref{omega1x}), the diameter of the $v$-sphere $S_v(0,M)$ is
$$v_{\Bn}(-x,x)=2\arcsin \frac{|x|}{\sqrt{1+|x|^2}}.$$
Hence it follows that $v_{\Bn}(-x, x)<2M$ and therefore the diameter of $S_v(0,M)$ is less than twice of the radius $M$.
\end{rem}

\medskip

\begin{nonsec}{\bf Hyperbolic metric.} The explicit formulas of the hyperbolic metric are as follows:
\be\label{cosh}
{\rm ch}\rho_{\Hn}(x,y)=1+\frac{|x-y|^2}{2d(x, \bdr \Hn)d(y, \bdr \Hn)}
\ee
for all $x,y\in \Hn$ \cite[p.35]{be1}, and
\be\label{sinh}
{\rm sh}\frac{\rho_{\Bn}(x,y)}{2}=\frac{|x-y|}{\sqrt{1-|x|^2}\sqrt{1-|y|^2}}
\ee
for all $x,y\in \Bn$ \cite[p.40]{be1}.
The hyperbolic metric is invariant under M\"obius transformations of $G$ onto $G'$ for $G,\,G'\in\{\Bn,\Hn\}$.
Hyperbolic geodesics are arcs of circles which are orthogonal to the boundary of the domain.
The problem of finding the midpoint of a hyperbolic segment has been studied in \cite{vw}.

\end{nonsec}

\medskip

\begin{nonsec}{\bf Distance ratio metric.}
For a proper open subset $G \subset {\mathbb R}^n\,$ and for all
$x,y\in G$, the  distance ratio
metric $j_G$ is defined as
\begin{eqnarray*}
 j_G(x,y)=\log \left( 1+\frac{|x-y|}{\min \{d(x,\partial G),d(y, \partial G) \} } \right)\,.
\end{eqnarray*}
The distance ratio metric was introduced by Gehring and Palka
\cite{gp} and in the above simplified form by  Vuorinen \cite{vu1}. Both definitions are
frequently used in the study of hyperbolic type metrics \cite{himps}, geometric theory of functions \cite{vu2}, and quasiconformality in Banach spaces \cite{va2}.

\end{nonsec}

\medskip

\begin{nonsec}{\bf Quasihyperbolic metric.}
Let $G$ be a proper subdomain of ${\mathbb R}^n\,$. For all $x,\,y\in G$, the quasihyperbolic metric $k_G$ is defined as
$$k_G(x,y)=\inf_{\gamma}\int_{\gamma}\frac{1}{d(z,\partial G)}|dz|,$$
where the infimum is taken over all rectifiable arcs $\gamma$ joining $x$ to $y$ in $G$ \cite{gp}.
The quasihyperbolic metric has found many applications in geometric function theory \cite{va2,vu2}. For some geometric properties of this metric, see \cite{k,l2}.

\end{nonsec}

It is well known that the following comparison relations hold for $x,y\in\Bn$,
\begin{equation}\label{jrho}
\frac{1}{2}\rho_\Bn(x,y)\leq j_\Bn(x,y)\leq \rho_\Bn(x,y),
\end{equation}

\begin{equation}\label{jk}
\frac{1}{2}\rho_\Bn(x,y)\leq k_\Bn(x,y)\leq \rho_\Bn(x,y).
\end{equation}
\medskip

\begin{nonsec} {\bf Lipschitz mappings.}
 Let $(X,d_X)$ and $(Y,d_Y)$ be metric spaces. Let $f: X\rightarrow Y$ be continuous and let $L\geq1$. We say that $f$ is $L$-lipschitz if
\begin{eqnarray*}
d_Y(f(x),f(y))\leq L d_X(x,y),\,\,{\rm for}\, x,\,y\in X,
\end{eqnarray*}
and $L$-bilipschitz if $f$ is
a homeomorphism and
\begin{eqnarray*}
d_X(x,y)/L\leq d_Y(f(x),f(y))\leq L d_X(x,y),\,\,{\rm for} \,x,\,y\in X.
\end{eqnarray*}
A $1$-bilipschitz mapping is called an isometry.
\end{nonsec}

\medskip

\begin{nonsec}{\bf Linear dilatation.}
Let $f: G\rightarrow\Rn$ be a continuous discrete function (i.e. the set $f^{-1}(y), y \in fG,$ consists of isolated points), where $G$ is a domain of $\Rn$. The linear dilatation of $f$
at a point $x\in G$ is given by
$$H(x,f)=\lim\sup_{r\rightarrow 0^+}\frac{L(x,f,r)}{l(x,f,r)},$$
where
$$L(x,f,r)=\max_z\{|f(z)-f(x)|: |z-x|=r\}\,,$$
$$l(x,f,r)=\min_z\{|f(z)-f(x)|: |z-x|=r\}\,,$$
for $r\in(0,d(x))$.
\end{nonsec}

\medskip

\begin{nonsec} {\bf Quasiregular mappings.}
Let $G_1$ and $G_2$ be domains in $\Rn$. A non-constant mapping $f: G_1\rightarrow G_2$ is said to be quasiregular if it satisfies the following conditions:

(1) $f$ is sense-preserving continuous, discrete, and open;

(2) $H(x,f)$ is locally bounded in $G_1$;

(3) There exists $a<\infty$ such that $H(x,f)\leq a$ for a.e. $x\in G_1\setminus B_f$, where $B_f$ is the branch set of $f$.

A quasiregular homeomorphism is called quasiconformal. Hence, a homeomorphism $f$ is quasiconformal if and only if $H(x,f)$ is bounded, see \cite[Theorem 4.13]{mrv}.
\end{nonsec}


\begin{nonsec}{\bf Automorphisms of $\Bn$.} We denote $a^*=\frac{a}{|a|^2}$ for $a\in\Rn\setminus\{0\}$, and $0^*=\infty$, $\infty^*=0$. For a fixed $a\in\Bn\setminus\{0\}$, let
$$
\sigma_a(z)=a^*+r^2(x-a^*)^*,\,\,r^2=|a|^{-2}-1
$$
be the inversion in the sphere $S^{n-1}(a^*,r)$ orthogonal to $S^{n-1}$. Then $\sigma_a(a)=0$,  $\sigma_a(a^*)=\infty$.

Let $p_a$ denote the reflection in the $(n-1)$-dimensional hyperplane $P(a,0)$ and define a sense-preserving M\"obius transformation by
\be\label{ta}
T_a=p_a\circ \sigma_a.
\ee
Then, $T_a\Bn=\Bn$, $T_a(a)=0$, and $T_a(e_a)=e_a$, $T_a(-e_a)=-e_a$. For $a=0$ we set $T_0=id$, where $id$ stands for the identity map. It is easy to see that $(T_a)^{-1}=T_{-a}$. Let $\mathcal{GM}(\Bn)$ be the group of M\"obius transformations which map $\Bn$ onto itself. It is well-known that for a given
$g\in\mathcal{GM}(\Bn)$
there is an orthogonal map $k$ such that $g=k\circ T_a$, where $a=g^{-1}(0)$  \cite[Theorem 3.5.1]{be1}. For $x\,,y\in\Bn$, $|T_x(y)|={\rm th}\frac12\rho_{\Bn}(x,y)$ \cite[(2.25)]{vu2}.
\end{nonsec}

The next lemma, so-called {\em monotone form of l'H${\rm \hat{o}}$pital's rule}, has found recently numerous applications in proving inequalities. See the extensive bibliography of \cite{avz}.

\begin{lem}\label{lhr}{\rm \cite[Theorem 1.25]{avv1}}
Let $-\infty<a<b<\infty$, and let $f,\,g: [a,b]\rightarrow \mathbb{R}$ be continuous on $[a,b]$, differentiable on $(a,b)$. Let $g'(x)\neq 0$ on $(a,b)$.Then, if $f'(x)/g'(x)$ is increasing(decreasing) on $(a,b)$, so are
\begin{eqnarray*}
\frac{f(x)-f(a)}{g(x)-g(a)}\,\,\,\,\,\,\,and\,\,\,\,\,\,\,\,\frac{f(x)-f(b)}{g(x)-g(b)}.
\end{eqnarray*}
If $f'(x)/g'(x)$ is strictly monotone, then the monotonicity in the conclusion is also strict.
\end{lem}

\bigskip

\section{The Schwarz-type Lemma}

In this section, we will prove the result related to the Schwarz-type lemma of the visual angle metric under $K$-quasiregular maps of the unit disk.

We first study the properties of some special functions involved in the proof of Theorem \ref{vs}. For $r\in(0,1)$ and $K>0$, we define the distortion function
$$\varphi_K(r)=\mu^{-1}(\mu(r)/K),$$
where $\mu(r)$ is the modulus of the planar Gr\"otzsch ring, see \cite[Exercise 5.61]{vu2}. It is clear that  $\varphi_1(r)=r$.
From now on we let $r'=\sqrt{1-r^2}$. Let $\K$ and $\E$ be the complete elliptic integral of the first and second kind, respectively.
The following derivative formulas of $\varphi_K(r)$ hold for $r\in (0,1)$, $K\in(0,\infty)$ \cite[(10.6)]{avv1}:
$$\frac{\partial \varphi_K(r)}{\partial r}=\frac 1K\frac{ss'^2\mathcal{K}(s)^2}{rr'^2\mathcal{K}(r)^2},$$
$$\frac{\partial \varphi_K(r)}{\partial K}=\frac{4}{\pi^2 K^2} ss'^2\mathcal{K}(s)^2\mu(r),$$
where $s=\varphi_K(r)$.
By \cite[Corollary]{wzc}, the function $\varphi_K(r)$ is concave in $r$ if $K>1$ and convex  if $0<K<1$.

\begin{lem}\label{vs1}
For $r\in (0,1)$, $K>1$, let $s=\varphi_K(r)$. Then the function\\
(1) $f_1(r)\equiv r^{-1/K}s$ is strictly decreasing from $(0,1)$ onto $(1, 4^{1-1/K})$;\\
(2) $f_2(r)\equiv\frac{s'\mathcal{K}(s)^2}{r'\mathcal{K}(r)^2}$ is strictly decreasing from $(0,1)$ onto $(0,1)$;\\
(3) $f_3(r)\equiv\sqrt{r'}\K(r)$ is decreasing from $[0,1)$ onto $(0,\pi/2]$;\\
(4) $f_4(r)\equiv\frac sr$ is strictly decreasing from $(0,1)$ onto $(1,\infty)$;\\
(5) $f_5(r)\equiv\frac{r}{\arctan(r/r')}$ is strictly decreasing from $(0,1)$ onto $(2/\pi,1)$;\\
(6) $f_6(r)\equiv 2\E(r)-r'^2\K(r)$ is strictly increasing from $(0,1)$ onto $(\pi/2,2)$.
\end{lem}

\begin{proof}
(1) This is a result from \cite[Lemma 10.9(1)]{avv1}.

(2) This is a result from \cite[Lemma 10.7(1)]{avv1}.

(3) This is a result from \cite[Lemma 3.21(7)]{avv1}.

(4) Since
$$f_4(x)=\frac s{r^{1/K}}\cdot\frac 1{r^{1-1/K}},$$
$f_4$ is strictly decreasing by (1). The limiting values are clear.

(5) Let $f_{51}(r)=r$ and $f_{52}(r)=\arctan(r/r')$. Then $f_{51}(0^+)=f_{52}(0^+)=0$.
By differentiation, we have
$$\frac{f'_{51}(r)}{f'_{52}(r)}=r',$$
which is strictly decreasing. Hence, by Lemma \ref{lhr} $f_5$ is strictly decreasing. The limiting value $f_5(0^+)=1$ follows by l'H\^opital Rule and $f_5(1^-)=2/\pi$ is clear.

(6) This is a result from \cite[Exercise 3.43(13)]{avv1}.

\end{proof}

\begin{lem}\label{vs2}
For $r\in (0,1)$, $K>1$, let $s=\varphi_K(r)$. Then the function\\
(1) $f_1(r)\equiv\frac{\arctan(s/s')}{\arctan(r/r')}$ is strictly decreasing from $(0,1)$ onto $(1,\infty)$;\\
(2) $f_2(r)\equiv\frac{\arctan(s/s')}{(\arctan(r/r'))^{1/K}}$ is strictly decreasing from $(0,1)$ onto $((\pi/2)^{1-1/K}, 4^{1-1/K})$.
\end{lem}
\begin{proof}
(1) Let $f_{11}(r)=\arctan(s/s')$ and $f_{12}(r)=\arctan(r/r')$. Then $f_{11}(0^+)=f_{12}(0^+)=0$.
By differentiation, we have
$$\frac{f'_{11}(r)}{f'_{12}(r)}=\frac1K\cdot\frac sr\cdot\frac{s'\mathcal{K}(s)^2}{r'\mathcal{K}(r)^2},$$
which is strictly decreasing by Lemma \ref{vs1}(2)(4). Therefore, by Lemma \ref{lhr} $f_1$ is strictly decreasing. The limiting value $f_1(0^+)=\infty$ follows by l'H\^opital Rule and $f_1(1^-)=1$ is clear.

(2) By logarithmic differentiation in $r$,
$$\frac{f'_2(r)}{f_2(r)}=\frac 1K \cdot\frac{1}{rr'^2\K(r)^2}\left(\frac{ss'\K(s)^2}{\arctan(s/s')}-\frac{rr'\K(r)^2}{\arctan(r/r')}\right).$$
By Lemma \ref{vs1}(3)(5), the function $\frac{rr'\K(r)^2}{\arctan(r/r')}$ is strictly decreasing on $(0,1)$. Since $s>r$, $f'_2(r)<0$ and hence $f_2$ is strictly decreasing. The limiting value $f_2(1^-)=(\pi/2)^{1-1/K}$ is clear. By Lemma \ref{vs1}(1)
$$f_2(0^+)=\lim_{r\rightarrow 0^+}\frac{\arctan(s/s')}{(\arctan(r/r'))^{1/K}}=\lim_{r\rightarrow 0^+}\frac{s/s'}{\left(r/r'\right)^{1/K}}=\lim_{r\rightarrow 0^+}\frac{s}{r^{1/K}}\cdot\frac{r'^{1/K}}{s'}=4^{1-1/K}.$$
\end{proof}

\begin{lem}\label{vs3}
For $K\ge 1$, let $r_0=\frac{\tan 1}{\sqrt{1+\tan^21}}\approx0.841471$ and $s_0=\varphi_K(r_0)$. Then
the function
$$f(K)\equiv4^{1-1/K}\cdot\frac{\arctan(r_0/r'_0)}{\arctan(s_0/s'_0)}$$
is strictly increasing. In particular,
\be\label{ck}
\max\left\{4^{1-1/K},\, \frac{\arctan(s_0/s_0')}{\arctan(r_0/r'_0)}\right\}=4^{1-1/K}.
\ee
\end{lem}

\begin{proof}
By logarithmic differentiation in $K$,
$$\frac{f'(K)}{f(K)}=\frac {\log 4}{K^2}\left(1-g(K)\right),$$
where
$$g(K)=\frac{4\mu(r_0)}{\pi^2\log 4}\cdot\frac{s_0}{\arctan(s_0/s'_0)}\cdot\left(\sqrt{s'_0}\K(s_0)\right)^2.$$
Since $s_0$ is increasing in $K$, $g$ is decreasing in $K$ by Lemma \ref{vs1}(3)(5). Then $g(K)<g(1)=\frac{4\mu(r_0)}{\pi^2\log 4}\cdot\frac{r_0 r'_0\K(r_0)^2}{\arctan(r_0/r'_0)}\approx0.744915<1$. Hence $f'(K)>0$, which implies that $f$ is strictly increasing. Therefore, (\ref{ck}) follows by the monotonicity of $f$ and $f(1)=1$.
\end{proof}

\begin{lem}\label{mthm1}{\rm \cite[Theorem 1.1]{klvw}}
For $G\in\{\Bn,\Hn\}$ and $x\,,y\in G$, let $\rho_{G}^*(x,y)=\arctan\left({\rm sh}\frac{\rho_G(x,y)}{2}\right)$. Then
$$\rho_{G}^*(x,y)\leq v_G(x,y)\leq 2\rho_{G}^*(x,y).$$
The left-hand side of the inequality is sharp and the constant 2 in the right-hand side of the inequality is the best possible.
\end{lem}

\begin{lem}\label{cgqm1}{\rm \cite[Theorem 11.2]{vu2}}
Let $f:\B\rightarrow\R^2$ be a non-constant $K$-quasiregular mapping with $f\B\subset\B$.
Then for all $x,\,y\in\B$,
$${\rm th}\frac12\rho_{\B}\left(f(x),f(y)\right)\le \varphi_K\left({\rm th}\frac12\rho_{\B}(x,y)\right).$$
\end{lem}

 We are now in a position to prove Theorem \ref{vs}.
\begin{proof}[Proof of Theorem \ref{vs}]
By Lemma \ref{mthm1} and Lemma \ref{cgqm1}, we get
\begin{eqnarray*}
A\equiv\frac{v_{\B}(f(x),f(y))}{\max\left\{v_{\B}(x,y),\,v_{\B}(x,y)^{1/K}\right\}}&\le&
\frac{2\arctan({\rm sh}\frac{\rho_{\B}(f(x),f(y))}{2})}{\max\left\{\arctan({\rm sh}\frac{\rho_{\B}(x,y)}{2}),\,(\arctan({\rm sh}\frac{\rho_{\B}(x,y)}{2}))^{1/K}\right\}}\\
&\le& \frac{2\arctan(s/s')}{\max\left\{\arctan(r/r'),\,(\arctan(r/r'))^{1/K}\right\}},
\end{eqnarray*}
where $r={\rm th}\frac{\rho_{\B}(x,y)}{2}$ and $s=\varphi_K(r)$. By Lemma \ref{vs2} and Lemma \ref{vs3},
$$A\le 2\max\{f_1(r_0),\,f_2(0^+)\}=2\cdot4^{1-1/K},$$
where $f_1$, $f_2$ are as in Lemma \ref{vs2} and $r_0$ is as in Lemma \ref{vs3}.

This completes the proof of Theorem \ref{vs}.
\end{proof}

In \cite{bv}, an explicit form of the Schwarz lemma for quasiregular mappings was given.
\begin{thm}\cite[Theorem 1.10]{bv}\label{bvtrans}
If $f:\B\rightarrow \R^2$ is a non-constant $K$-quasiregular map with $f\B\subset\B$ and $\rho$ is the hyperbolic metric of $\B$, then
$$\rho_{\B}(f(x),f(y))\le c(K) \max\{\rho_{\B}(x,y),\,\rho_{\B}(x,y)^{1/K}\}$$
for all $x,\,y\in\B$, where $c(K)=2\arth(\varphi_K({\rm th}\frac{1}{2}))$ and, in particular, $C(1)=1$.
\end{thm}
The proof of this theorem involves the following monotonicity of the transcendental function $\varphi_K(r)$, see \cite[Lemma 4.8]{bv}.
\begin{lem}
For $K>1$, the function
$$
g(r)\equiv\frac{\arth(\varphi_K(r))}{(\arth r)^{1/K}}
$$
is strictly increasing on $(0,1)$.
\end{lem}
There exists, however, a gap in the proof of the above result due to the using of a wrong claim that the function $r\mapsto \varphi_K(r)/r$ is increasing on $(0,1)$. In fact, by \cite[Corollary]{wzc}, the function $\varphi_K(r)$ is concave in $r$ if $K>1$, and hence the function $r\mapsto \varphi_K(r)/r$ is decreasing on $(0,1)$, see Lemma \ref{vs1}(4). We give a correction for the proof of the monotonicity of the function $g$ as follows.
\begin{proof}
Let $g_1(r)=r\K(r)^2/\arth r=g_{11}(r)/g_{12}(r)$, where $g_{11}(r)=r\K(r)^2$ and $g_{12}(r)=\arth r$. Then $g_{11}(0)=g_{12}(0)=0$, and
$$
\frac{g_{11}'(r)}{g_{12}'(r)}=\K(r)(2\E(r)-r'^2\K(r))\,,
$$
which is strictly increasing by Lemma \ref{vs1}(6) and implies that $g_1(r)$ is also strictly increasing by Lemma \ref{lhr}. Let $s=\varphi_K(r)$. Then we have that
\begin{equation}\label{inequal-correction1}
\frac{s\K(s)^2}{\arth s}-\frac{r\K(r)^2}{\arth r}>0
\end{equation}
since $s>r$ for all $K>1$ and $0<r<1$.
By logarithmic differentiation, we get that
$$
\frac{g'(r)}{g(r)}=\frac{1}{K}\frac{1}{rr'^2\K(r)^2}\left(\frac{s\K(s)^2}{\arth s}-\frac{r\K(r)^2}{\arth r}\right)>0\,,
$$
which implies that for given $K>1$, the function $g$ is strictly increasing on $(0,1)$.
\end{proof}

The following result gives the sharp distortion of the distance ratio metric and  the quasihyperbolic metric under quasiregular mappings of the unit disk into itself. A similar result has been obtained in \cite[Theorem 1.8]{kvz} for the higher dimensional case.

\begin{cor}
If $f:\B\rightarrow \R^2$ is a non-constant $K$-quasiregular map with $f\B\subset\B$, then
$$j_{\B}(f(x),f(y))\le 2c(K) \max\{j_{\B}(x,y),\,j_{\B}(x,y)^{1/K}\}$$
and
$$k_{\B}(f(x),f(y))\le 2c(K) \max\{k_{\B}(x,y),\,k_{\B}(x,y)^{1/K}\}$$
for all $x,\,y\in\B$, where $c(K)$ is the same as in Theorem \ref{bvtrans}.
\end{cor}

\begin{proof}
For the distance ratio metric, we combine the comparison relation \eqref{jrho} with Theorem \ref{bvtrans} to get that, for $x,y\in \B$ and $x\neq y$,
\begin{align*}
\frac{j_\B(f(x),f(y))}{\max\{j_\B(x,y),j_\B(x,y)^{1/K}\}} & \leq \frac{\rho_\B(f(x),f(y))}{\max\{\frac{1}{2}\rho_\B(x,y), \frac{1}{2^{1/K}}\rho_\B(x,y)^{1/K}\}}\\
                                                          & = \frac{2\rho_\B(f(x),f(y))}{\max\{\rho_\B(x,y), 2^{1-1/K}\rho_\B(x,y)^{1/K}\}}\\
                                                          & \leq \frac{2\rho_\B(f(x),f(y))}{\max\{\rho_\B(x,y), \rho_\B(x,y)^{1/K}\}}\\
                                                          & \leq 2c(K).
\end{align*}
By use of the inequality \eqref{jk} and a similar argument with the distance ratio metric, we get the result for the quasihyperbolic metric.
\end{proof}


\bigskip


\section{Convex domains and bilipschitz maps }

The following theorem shows the comparison of the visual angle metric and the distance ratio metric on convex domains, which is the main tool to prove Theorem \ref{mthf}.
\begin{thm}\label{bcb}
Let $G$ be a proper convex subdomain of $\Rn$. Let $x,\,y\in G$ and $t=e^{j_G(x,y)}-1$. 
Then
$$\arcsin\frac{t}{t+2}\le v_G(x,y)\le 2\arcsin \frac{t}{\sqrt{4+t^2}}.$$
The left-hand side of the inequality  is sharp and the constant $2$ in the right-hand side of the inequality is the best possible.
\end{thm}

\begin{figure}[ht]
  \begin{center}
    \includegraphics[height=6cm]{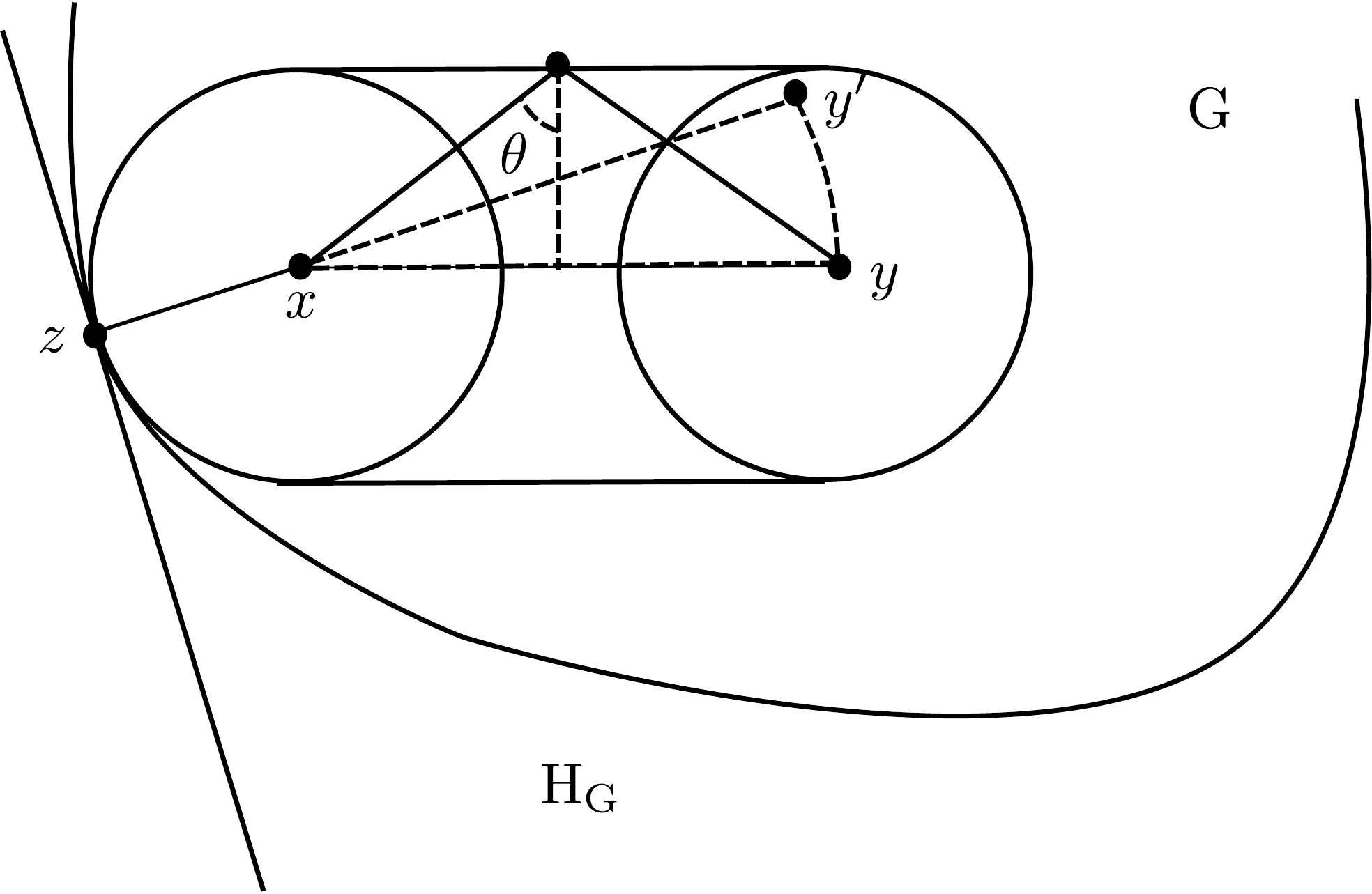}
  \end{center}
  \caption{\label{cv} }
\end{figure}
\begin{proof}

Without loss of generality, we may assume that $d(x)\le d(y)$. Then $t=\frac{|x-y|}{d(x)}$. Choose $z\in \partial G$ such that $d(x)=|x-z|$. Let $y'\in {\rm ray}(x, x-z)$ with $|y-x|=|y'-x|$. Let $H_G,\, x \in H_G\,,$ be the half space whose boundary is orthogonal to $[x,z]$ at the point $z$. It is easy to see that $G\subset H_G$. Let $B_{xy}$ be the convex hull of $B^n(x,d(x)) \cup B^n(y,d(x))\,,$ see Figure \ref{cv}.
By the monotonicity property of $v_G$, we have
$$v_{H_G}(x,y)\le v_G(x,y)\le v_{B_{xy}}(x,y),$$
where
$$v_{H_G}(x,y)\ge v_{H_G}(x,y')=\arcsin\frac{|x-y|/2}{d(x)+|x-y|/2}=\arcsin\frac{t}{t+2}\,,$$
and
$$v_{B_{xy}}(x,y)= 2\theta=2\arcsin\frac{|x-y|/2}{\sqrt{d(x)^2+(|x-y|/2)^2}}=2\arcsin\frac{t}{\sqrt{4+t^2}}.$$

For the sharpness of the left-hand side of the inequality, we consider the domain $\Hn$ and two points $x\,,y\in \Hn$ with $L(x,y-x)$ perpendicular to the boundary $\partial \Hn$.

For the sharpness of the right-hand side of the inequality, we consider the domain $\Hn$ and two points $x\,,y\in \Hn$ with $L(x,y-x)$ parallel to the  boundary $\partial \Hn$.


This completes the proof.
\end{proof}

\begin{proof}[Proof of Theorem \ref{mthf}]
Let $t=e^{j_{G_1}(x,y)}-1$. Given $\varepsilon>0$, for small enough $t>0$, by Theorem \ref{bcb} we have
$$\frac{t}{2(1+\varepsilon)}\le\arcsin\frac{t}{t+2}\le v_{G_1}(x,y)\le 2\arcsin \frac{t}{\sqrt{4+t^2}}\le 2\arcsin \frac{t}{2}\le (1+\varepsilon) t\,.$$
Then we have
\begin{eqnarray*}
\frac{|f(x)-f(y)|}{|f(x)-f(z)|}
&\le&\frac{2(1+\varepsilon)\,\min\{d(f(x))\,,d(f(y))\}\,v_{G_2}(f(x),f(y))}{ \frac 1{1+\varepsilon}\,\min\{d(f(x))\,,d(f(z))\}\,v_{G_2}(f(x),f(z))}\\
&\le&\frac{2(1+\varepsilon)^2\,\min\{d(f(x))\,,d(f(y))\}L\,v_{G_1}(x,y)}{\min\{d(f(x))\,,d(f(z))\}\,v_{G_1}(x,z)/L}\\
&\le&\frac{2(1+\varepsilon)^2\,\min\{d(f(x))\,,d(f(y))\}\,L(1+\varepsilon)\,\frac{|x-y|}{\min\{d(x)\,,d(y)\}}}
{\min\{d(f(x))\,,d(f(z))\}\,\frac{1}{2(1+\varepsilon)L}\,\frac{|x-z|}{\min\{d(x)\,,d(z)\}}}\\
&=&4(1+\varepsilon)^4\,L^2\cdot\frac{\min\{d(f(x))\,,d(f(y))\}}{\min\{d(f(x))\,,d(f(z))\}}\cdot\frac{\min\{d(x)\,,d(z)\}}{\min\{d(x)\,,d(y)\}}\cdot\frac{|x-y|}{|x-z|}\\
&\rightarrow&4(1+\varepsilon)^4\,L^2\,,\quad\quad{\rm as}\quad |x-y|=|x-z|=r\rightarrow 0.
\end{eqnarray*}
Since $\varepsilon$ is arbitrary, we have $H(x,f)\le 4L^2$. Hence $f$ is quasiconformal and with linear dilatation at most $4L^2$.
\end{proof}

\begin{rem}
By Theorem \ref{mthf}, a bilipschitz map from the unit ball onto itself with respect to the visual angle metric is a quasiconformal map. However, the converse is not true.

For example, let $a\in(0,1)$ and $z\in \Bn$. The radial map $f(z)=z/|z|^{1-a}$ is quasiconformal \cite[16.2]{va1}. But,
if putting $|x|=|y|$, $2\theta=\ang(x,0,y)>0$, then by (\ref{omega1x}) we have
\begin{eqnarray*}
\lim_{x\rightarrow0}\frac{v_{\Bn}(f(x),f(y))}{v_{\Bn}(x,y)}=\lim_{x\rightarrow0}\frac{\arctan\frac{|x|^a\sin\theta}{1-|x|^a\cos\theta}}{\arctan\frac{|x|\sin\theta}{1-|x|\cos\theta}}
=\lim_{x\rightarrow0} |x|^{-(1-a)}\frac{1-|x|\cos\theta}{1-|x|^a\cos\theta}=\infty.
\end{eqnarray*}
\end{rem}

\begin{lem}\label{lerho1}
Let $L\in[1,\infty)$ and $\varepsilon\in(0,1)$. Then\\
(1) $f_1(r)\equiv \frac{\arcsin r}{\log(1+r)}$ is strictly increasing from $(0, 1)$ onto $(1, \pi/\log 4)$;\\
(2) $f_2(r)\equiv \frac{\sin(4L r)}{\sin r}$ is strictly decreasing from $(0,\frac{\pi}{8L})$ onto $((\sin\frac{\pi}{8L})^{-1}\,,4L)$;\\
(3) $f_3(r)\equiv \frac{{\rm arth}(4L r)}{{\rm arth} r}$ is strictly increasing from $(0, \frac{\varepsilon}{4L})$ onto $(4L\,,{\rm arth} \varepsilon/{\rm arth}\frac{\varepsilon}{4L})$.
\end{lem}

\begin{proof}
(1) Let $f_1(r)=\frac{f_{11}(r)}{f_{12}(r)}$ with $f_{11}(r)=\arcsin r$ and $f_{12}(r)=\log (1+r)$. It is clear that $f_{11}(0^+)=f_{12}(0^+)=0$.
By differentiation,
$$\frac{f'_{11}(r)}{f'_{12}(r)}=\sqrt{\frac{1+r}{1-r}},$$
which is increasing on $(0,1)$. Therefore, $f_1$ is strictly increasing by Lemma \ref{lhr}. The limiting value $f_1(1^-)=\pi/\log 4$ is clear and $f_1(0^+)=1$ by l'H\^opital's Rule.

(2) Let $f_{21}(r)=\sin(4L r)$ and $f_{22}(r)=\sin r$. It is clear that $f_{21}(0^+)=f_{22}(0^+)=0$. Then
$$\frac{f'_{21}(r)}{f'_{22}(r)}=4L\frac{\cos(4L r)}{\cos r}\equiv 4L h(r).$$
Since
$$\cos^2r h'(r)=-4L\sin((4L-1)r)-(4L-1)\sin r\cos(4L r)<0,$$
by Lemma \ref{lhr}, $f_2$ is strictly decreasing on $(0,\frac{\pi}{8L})$. The limiting value $f_2(\frac{\pi}{8L})=(\sin\frac{\pi}{8L})^{-1}$ is clear and $f_2(0^+)=4L$ by l'H\^opital Rule.

(3) Let $f_{31}(r)={\rm arth}(4L r)$ and $f_{32}(r)={\rm arth}\, r$. It is clear that $f_{31}(0^+)=f_{32}(0^+)=0$. Then
$$\frac{f'_{31}(r)}{f'_{32}(r)}=4L\frac{1-r^2}{1-(4L r)^2}=\frac 1{4L}\left(1+\frac{(4L)^2-1}{1-(4L r)^2}\right),$$
which is strictly increasing on $(0, \frac{\varepsilon}{4L})$. Therefore, $f_3$ is strictly increasing on $(0, \frac{\varepsilon}{4L})$
by Lemma \ref{lhr}. The limiting value $f_3( \frac{\varepsilon}{4L})={\rm arth} \varepsilon/{\rm arth}\frac{\varepsilon}{4L}$ is clear and $f_3(0^+)=4L$ by l'H\^opital Rule.

\end{proof}

\begin{thm}\label{vk}
Let $G$ be a proper convex subdomain of $\Rn$. For all $x,\,y\in G$, there holds
$$v_G(x,y)\le c k_G(x,y),$$
where $c=\frac{\pi}{\log 4}\approx 2.26618$.
\end{thm}
\begin{proof}
For arbitrary $x,y\in G$ such that $x\neq y$, the quasihyperbolic geodesic segment with end points $x,y$ is denoted by $J_k[x,y]$. Select points $\{z_i\}_{i=0}^n\in J_k[x,y]$ with $z_0=x$, $z_n=y$ such that $t_i=\frac{|z_{i+1}-z_i|}{\min\{d(z_i), d(z_{i+1})\}}\in (0,1)$,
where $i=0,\cdots,n-1$.

By Theorem \ref{bcb} and Lemma \ref{lerho1}(1), we have
\begin{eqnarray*}
v_G(z_i, z_{i+1})&\le& 2\arcsin\frac{t_i}{\sqrt{4+t_i^2}}\le2\arcsin\frac {t_i}2\le \arcsin t_i\\
&\le& \frac{\pi}{\log 4} j_G(z_i, z_{i+1})\le \frac{\pi}{\log 4} k_G(z_i, z_{i+1}).
\end{eqnarray*}

Then
$$
v_G(x,y)\leq\sum_{i=0}^{n-1}v_G(z_i, z_{i+1})\leq \frac{\pi}{\log 4} \sum_{i=0}^{n-1}k_G(z_i, z_{i+1})=\frac{\pi}{\log 4}k_G(x,y).
$$

This completes the proof.
\end{proof}

\medskip

\begin{lem}\label{lerho2}
Let $f: \Bn\rightarrow \Bn$ be an $L$-bilipschitz map with respect to the visual angle metric.
Let $\varepsilon\in(0,1)$ and $c(\varepsilon)={\rm arth} \varepsilon/{\rm arth}\frac{\varepsilon}{4L}$. Furthermore, let $x\,,y\in\Bn$ and satisfy \be\label{rhocon}
{\rm th}\frac{\rho_{\Bn}(x,y)}{2}\leq \min\left\{\frac{\varepsilon}{4L}\,,\sin\frac{\pi}{8L}\right\}.
\ee
Then $$\rho_{\Bn}(f(x),f(y))\leq c(\varepsilon) \rho_{\Bn}(x,y).$$
\end{lem}
\begin{proof}
Let $x,y\in\Bn$ satisfy (\ref{rhocon}) such that $x\neq y$. 
By Lemma \ref{mthm1}, it is not difficult to obtain that $g=T_{f(x)}\circ f\circ T_{-x} : \Bn\rightarrow\Bn$ is a $4L$-bilipschitz map with respect to the visual angle metric and $g(0)=0$. Then by (\ref{omega0x})
$$
\arcsin|T_{f(x)}(f(y))|=v_{\Bn}(g(T_x(x)),g(T_x(y)))\leq 4L v_{\Bn}(T_x(x),T_x(y))=4L\arcsin|T_x(y)|.
$$
Since $\arcsin|T_x(y)|=\arcsin{\rm th}\frac{\rho(x,y)}{2}\leq\frac{\pi}{8L}$, by Lemma \ref{lerho1}(2) we have
$$
|T_{f(x)}(f(y))|\leq \sin(4L\arcsin|T_x(y)|)\leq 4L |T_x(y)|.
$$
Hence by Lemma \ref{lerho1}(3)
$$
\rho_{\Bn}(f(x),f(y))\leq 2 {\rm arth}\left(4L{\rm th}\frac{\rho_{\Bn}(x,y)}{2}\right)\leq c(\varepsilon)\rho_{\Bn}(x,y).
$$
\end{proof}

In the following we will show that for $G,\,G'\in\{\Bn,\Hn\}$, if $f:G\rightarrow G'$ is a bilipschitz map with respect to the visual angle metric, then $f$ is bilipschitz with respect to the hyperbolic metric, too.

\begin{thm}\label{thrho}
Let $f: \Bn\rightarrow \Bn=f(\Bn)$ be an $L$-bilipschitz map with respect to the visual angle metric. Then $f$ is a $4L$-bilipschitz map with respect to the hyperbolic metric.
\end{thm}
\begin{proof}
For arbitrary $x,y\in\Bn$ such that $x\neq y$, the hyperbolic geodesic segment joining the points $x,y$ is denoted by $J_\rho[x,y]$. Select points $\{z_k\}_{k=0}^n\in J_\rho[x,y]$ with $z_0=x$, $z_n=y$ such that $\rho_{\Bn}(z_k,z_0)>\rho_{\Bn}(z_{k-1},z_0)$ and
$$
{\rm th}\frac{\rho_{\Bn}(z_{k-1},z_k)}{2}\leq \min\left\{\frac{\varepsilon}{4L}\,,\sin\frac{\pi}{8L}\right\},
$$
where $k=1,\cdots,n$ and $\varepsilon\in(0,1)$ is a constant.
By Lemma \ref{lerho2},
$$
\rho_{\Bn}(f(x),f(y))\leq\sum_{k=1}^n\rho_{\Bn}(f(z_{k-1}),f(z_k))\leq c(\varepsilon)\sum_{k=1}^n\rho_{\Bn}(z_{k-1},z_k)=c(\varepsilon) \rho_{\Bn}(x,y),
$$
where $c(\varepsilon)$ is as in Lemma \ref{lerho2}. Then, letting $\varepsilon$ tend to $0$, we obtain
$$
\rho_{\Bn}(f(x),f(y))\leq 4L \rho_{\Bn}(x,y).
$$
Applying the above argument to $f^{-1}$, we get
$$
\rho_{\Bn}(f^{-1}(x),f^{-1}(y))\leq 4L \rho_{\Bn}(x,y)
$$
and hence
$$
\rho_{\Bn}(f(x),f(y))\geq \frac{1}{4L} \rho_{\Bn}(x,y).
$$

This completes the proof.
\end{proof}

\begin{cor}\label{corth21}
Let $f: \Hn\rightarrow  \Hn=f( \Hn)$ be an $L$-bilipschitz map with respect to the visual angle metric. Then $f$ is a $16L$-bilipschitz map with respect to the hyperbolic metric.
\end{cor}
\begin{proof}
Let $g=g_1\circ f\circ g_2$ where $g_1: \Hn\rightarrow\Bn$ and $g_2: \Bn\rightarrow\Hn$ are two M\"obius transformations. Then $g:\Bn\rightarrow\Bn$ is a $4L$-bilipschitz map with respect to the visual angle metric by Lemma \ref{mthm1}.

Since $f=g^{-1}_1\circ g\circ g^{-1}_2$, by Theorem \ref{thrho} we conclude that $g$ is a $16L$-bilipschitz map with respect to the hyperbolic metric.
\end{proof}

Similarly, we have the following two corollaries.
\begin{cor}\label{corth22}
Let $f: \Hn\rightarrow \Bn=f( \Hn)$ be an $L$-bilipschitz map with respect to the visual angle metric. Then $f$ is an $8L$-bilipschitz map with respect to the hyperbolic metric.
\end{cor}

\begin{cor}\label{corth23}
Let $f: \Bn\rightarrow \Hn=f(\Bn)$ be an $L$-bilipschitz map with respect to the visual angle metric. Then $f$ is an $8L$-bilipschitz map with respect to the hyperbolic metric.
\end{cor}

\begin{proof}[Proof of Theorem \ref{mths}]
The result immediately follows from Theorem \ref{thrho}, Corollary \ref{corth21},  Corollary \ref{corth22} and Corollary \ref{corth23}.
\end{proof}

\medskip

\bigskip

\subsection*{Acknowledgments}
This research of both authors were supported by the Academy of Finland,
Project 2600066611. The first author was also supported by Turku University Foundation, the Academy of Finland, Project 268009, and Science Foundation of Zhejiang Sci-Tech University(ZSTU). The authors thank Dr. Xiaohui Zhang for useful discussion and helpful comments and the referee for valuable corrections.


\end{document}